\documentclass[fleqn,12pt]{article}
\usepackage{amsmath,amssymb,amsthm,xcolor,framed,esint,dsfont}
\usepackage[english]{babel}
\usepackage[margin=3.4cm]{geometry}
\usepackage{csquotes}
\usepackage[bookmarks]{hyperref}
\usepackage{tikz}
\numberwithin{equation}{section}

\def\XXint#1#2#3{{\setbox0=\hbox{$#1{#2#3}{\int}$}
		\vcenter{\hbox{$#2#3$}}\kern-.5\wd0}}
        
\newcommand{\R}{{\mathbb R}}

\newcommand{\Z}{{\mathbb Z}}
\newcommand{\T}{{\mathbb T}}
\newcommand{\C}{{\mathbb C}}

\newcommand{\e}{{\epsilon}}
\newcommand{\loc}{{\mathrm{loc}}}

\newcommand{\nn}{\nonumber}
\newcommand{\ep}{\epsilon}

\DeclareMathOperator{\dist}{dist}
\DeclareMathOperator{\supp}{supp}
\DeclareMathOperator{\dv}{div}

\newtheorem{theorem}{Theorem}[section]

\newtheorem{lemma}[theorem]{Lemma}

\theoremstyle{definition}

\newtheorem{remark}[theorem]{Remark}
\title{On Aviles-Giga limit states with $L^p$ entropy productions}

\date{}
\author{Xavier Lamy\footnote{Institut de Math\'ematiques de Toulouse, UMR 5219, Universit\'e de Toulouse, CNRS, UPS
		IMT, F-31062 Toulouse Cedex 9, France. Email: Xavier.Lamy@math.univ-toulouse.fr}
	\and Andrew Lorent\footnote{Department of Mathematical Sciences, University of Cincinnati, Cincinnati, OH 45221, USA. Email: lorentaw@uc.edu} 
	\and Guanying Peng\footnote{Department of Mathematical Sciences, Worcester Polytechnic Institute, Worcester, MA 01609, USA. Email: gpeng@wpi.edu}}

\begin{document}

\maketitle

\begin{abstract}
The Aviles-Giga energy provides  sequences of maps converging to weak solutions $m\colon\Omega \subset\mathbb R^2\to\mathbb R^2$
of the eikonal equation
\begin{align*}
\mathrm{div}\, m=0\text{ in }\mathcal D'(\Omega),\quad |m|=1\text{ a.e. in }\Omega\,,
\end{align*}
whose entropy productions $\mathrm{div}\,\Phi(m)$ are Radon measures in $\Omega$, controlled by the energy.
Here, the entropies are all $C^2$ vector fields $\Phi\colon\mathbb S^1\to\mathbb R^2$ such that $\dv\Phi(m_*)=0$ for any smooth solution $m_*$.
It is conjectured that the entropy production measures are concentrated on the one-dimensional jump set of $m$,
as follows from the chain rule if $m$ has bounded variation.
In particular, the entropy production measures should vanish if they coincide with $L^p$ functions: this is what we establish in this note if $p$ is not too small and under natural boundary conditions.
\end{abstract}

\section{Introduction}

Given a bounded domain $\Omega\subset\R^2$, we consider weak solutions $m: \Omega\to\R^2$ of the 2D eikonal equation given by
\begin{equation}\label{eq:eikonal}
	|m|=1 \text{ a.e. in }\Omega, \qquad \dv m = 0 \text{ in }\mathcal{D}'(\Omega). 
\end{equation}
If $m$ is a \emph{smooth}  solution of \eqref{eq:eikonal}, the chain rule ensures that
$\dv\Phi(m)=0$ 
for any smooth vector field $\Phi\colon \mathbb S^1\to\R^2$ 
with derivative tangent to $\mathbb S^1$. 
By analogy with hyperbolic conservation laws, 
these vector fields are called entropies.
We denote $C^2$ entropies by
\begin{align*}
\mathrm{ENT}=\bigg\lbrace
\Phi\in C^2(\mathbb S^1;\R^2)\colon
&
\exists\lambda_\Phi\in C^1(\mathbb S^1;\R),
\nonumber
\\
&
\frac{d}{d\theta}\Phi(e^{i\theta})=\lambda_\Phi(e^{i\theta})ie^{i\theta}
\quad\forall\theta\in\R
\bigg\rbrace\,.
\end{align*}
If $m$ is  a general weak solution, 
we expect the distributions $\dv\Phi(m)$ to carry information about singularities of $m$,
and we are particularly interested 
 in \emph{finite entropy solutions}, i.e., weak solutions of \eqref{eq:eikonal} satisfying the finite entropy production property
\begin{equation}\label{eq:finite_entropy}
	\dv \Phi(m)\in\mathcal{M}(\Omega)\qquad \forall \Phi\in\mathrm{ENT}\,,
\end{equation}
where $\mathcal M(\Omega)$ is the set of finite Radon measures on $\Omega$.
As shown in \cite{DKMO01}, this property is satisfied by limits $m=\lim_{\textcolor{black}{\ep\to 0}} m_\e$ of sequences of vector fields $m_\e\colon\Omega\to\R^2$ with $\dv m_\e=0$ and bounded Aviles-Giga energy
\begin{align}\label{eq:AG}
E_\e(m_\e)=\int_\Omega \left( \frac\e 2 |\nabla m_\e|^2 + \frac{1}{2\e} (1-|m_\e|^2)^2\right)\, dx\leq C\,,
\end{align}
and is also relevant in some micromagnetics models \cite{JOP02}.
The conjecture that these energy functionals $\Gamma$-converge to a line-energy functional concentrated on the one-dimensional jump set of $m$ has attracted sustained attention over the past decades \cite{AG87, AG99,JK00,ADM99,DKMO01,DLO03, IM12,marconi21}.
If true, that conjecture would imply a similar concentration property for the entropy productions, namely that
\begin{align}\label{eq:concentration_ent}
\dv\Phi(m)\ll \mathcal H^1 \lfloor J_m\,,
\end{align}
for any weak finite-entropy solution of \eqref{eq:eikonal}-\eqref{eq:finite_entropy}, 
where $J_m$ is the $1$-rectifiable jump set of $m$ (in the sense of \cite{delnin21}), 
see \cite[Conjecture~1]{DLO03}.
If $m\in BV(\Omega)$, the concentration property \eqref{eq:concentration_ent} follows from the chain rule,
but solutions of \eqref{eq:eikonal}-\eqref{eq:finite_entropy} are in general not of bounded variation, see \cite{ADM99}.
In fact, a weak solution of \eqref{eq:eikonal} satisfies \eqref{eq:finite_entropy} if and only if it has
the  Besov regularity $B^{1/3}_{3,\infty}$ \cite{GL20}.

A first major step towards \eqref{eq:concentration_ent} was established in \cite{DLO03} 
by showing that $J_m$ coincides ($\mathcal H^1$-a.e.) with points  of positive one-dimensional density for (at least one of) the measures \eqref{eq:finite_entropy},
hence  $\dv\Phi(m)\lfloor \Omega\setminus J_m$  vanishes on  sets with finite $\mathcal H^1$-measure.
Conjectures very similar to \eqref{eq:concentration_ent} were recently settled by E.~Marconi in the contexts of Burgers' equation \cite{marconi22} and of a related micromagnetics model \cite{marconi23}.
These methods also brought new information on the structure of  $\dv\Phi(m)\lfloor \Omega\setminus J_m$  \cite{marconi21,LM26},
but the validity of \eqref{eq:concentration_ent} remains open.

If true, the concentration property \eqref{eq:concentration_ent} 
would imply in particular
\begin{align*}
\Big(
\forall\Phi\in\mathrm{ENT}\,,
\;\dv\Phi(m)\in L^p(\Omega) \Big)
\;
\Longrightarrow
\;
\Big(
\forall\Phi\in\mathrm{ENT}\,,
\;
\dv\Phi(m)=0
\Big)
\,,
\end{align*}
for any $p\geq 1$,
and any weak solution $m$ of the eikonal equation \eqref{eq:eikonal}.
But even this
 weaker concentration property is not known.
Our main result establishes it
under natural boundary conditions if $p\geq 2$.
Note that $p=2$ is critical in the sense that the rescaling $m_r(x)=m(rx)$ yields 
\begin{align*}
\|\dv\Phi(m_r)\|_{L^p(B_1)}=r^{1-\frac 2p}\|\dv\Phi(m)\|_{L^p(B_r)}\,.
\end{align*}
If the domain is a disk, our result also holds for subcritical exponents $p\geq 1.95$.
\color{black}
More precisely, we prove the stronger implication
\begin{align*}
\Big(
\forall\Phi\in
\lbrace \Sigma_1,\Sigma_2\rbrace\,,
\;\dv\Phi(m)\in L^p(\Omega) \Big)
\;
\Longrightarrow
\;
\Big(
\forall\Phi\in\mathrm{ENT}\,,
\;
\dv\Phi(m)=0
\Big)
\,,
\end{align*}
where $\Sigma_1,\Sigma_2$ 
are the Jin-Kohn entropies introduced in \cite{JK00}.
This entropy pair plays a fundamental role in the 
$\Gamma$-convergence conjecture for the Aviles-Giga energy \eqref{eq:AG}.
The total variation of the $\R^2$-valued measure $\dv\Sigma(m)=(\dv\Sigma_1(m),\dv\Sigma_2(m))$
provides a lower bound 
\begin{align*}
\liminf_{\e\to 0} E_\e(m_\e) \geq |\dv\Sigma(m)|(\Omega)\,,
\end{align*}
for any sequence $m_\e\to m$ in $\mathcal D'(\Omega)$ with $\dv m_\e=0$ \cite{ADM99},
and that lower bound
 is sharp if $m\in BV(\Omega)$, 
in the sense that there exists $m_\e\to m$ whose energy converges to $|\dv\Sigma(m)|(\Omega)$
\cite{CDL07,poliakovsky07}.

\begin{remark}
The Jin-Kohn entropies are polynomial fields of degree three.
One way to represent them explicitly is
via the matrix 
 $\Sigma(e^{i\theta})\in \R^{2\times 2}$ 
 with rows 
 $\Sigma_1(e^{i\theta})$, $\Sigma_2(e^{i\theta})$.
This matrix is
characterized by its conformal-anticonformal decomposition
\begin{align*}
\Sigma(e^{i\theta})z=e^{i\theta}z +\frac 13 e^{3i\theta}\bar z\,,\qquad\forall  z\in\R^2\approx \C\,.
\end{align*}
Here, the left-hand side is a matrix multiplication for $z=(z_1,z_2)\in\R^2$ and the terms in the right-hand side are complex multiplications for $z=z_1+iz_2\in\C$.
\end{remark}
\color{black}

We impose tangential boundary conditions, which are natural in the context of the Aviles-Giga energy \cite{JK00}.
More precisely, we assume that $\Omega$ is a bounded simply connected open set of class $C^2$ and that the inner trace of $m$ along 
$\partial\Omega$,
which is well-defined thanks to the trace theorem of \cite{vasseur01} and the kinetic formulation of \cite{GL20}, 
is equal to the counterclockwise unit tangent $\tau\in C^1(\partial\Omega;\mathbb S^1)$.
The trace theorem of \cite{vasseur01} is stated in a slightly different context, 
but one may also avoid referring to traces and impose an outer boundary layer condition
\begin{align}\label{eq:tangent_layer}
m(x)=\tau(\pi_{\partial\Omega}(x))\quad\forall x\in \Omega_\delta\setminus\Omega\,,
\quad
\Omega_\delta=\lbrace x\in\R^2\colon \dist(x,\Omega)< \delta\rbrace\,,
\end{align}
with $\delta>0$ small enough that  the nearest-neighbor projection $\pi_{\partial\Omega}\colon \Omega_\delta\setminus\Omega\to\partial\Omega$ is $C^1$.
Under these boundary conditions, 
if $\dv\Phi(m)=0$ in $\Omega_\delta$ for all $\Phi\in\mathrm{ENT}$, then
$\Omega$ must be a disk and $m$ a vortex,
that is,
$m(x)=ix/|x|$ in centered coordinates, see \cite[Theorem~1.2]{JOP02}.
Thus our main result can be stated as follows.

%

\begin{theorem}\label{t:Lpent_tangent}
Let $\Omega\subset\R^2$ be a bounded, $C^2$, simply connected open set,
and $m\colon \Omega\to\R^2$ a weak solution of the eikonal equation \eqref{eq:eikonal}
subject to the tangential condition \eqref{eq:tangent_layer}.
\begin{itemize}
\item If $\dv\Sigma_k(m)\in L^p(\Omega_\delta)$ for $k=1,2$ and  $p=2$ then
 $\Omega$ is a disk and $m$ is a vortex.
\item If $\Omega$ is a disk and $\dv\Sigma_k(m)\in L^p(\Omega_\delta)$  for $k=1,2$  and  $p=1.95$, then $m$ is a vortex.
\end{itemize}
\end{theorem}

 Theorem~\ref{t:Lpent_tangent} 
is a consequence of results from our previous works \cite{LP23,LLP25}
%
and of the following new regularity result.

\begin{theorem}\label{t:besov_reg}
Let $\Omega\subset\R^2$ be a bounded open set,
and $m$ a weak solution of \eqref{eq:eikonal}.
Assume 
\begin{align}
&
m\in W^{1,1}(\Omega\setminus \textcolor{black}{\overline {\Omega'}};\mathbb S^1)
\quad\text{for some }\Omega'\subset\subset\Omega\,,
\label{eq:layer}
\\
\text{and }
&
\dv\Phi(m)\in L^p(\Omega)\qquad\text{for all {\color{black}odd} } \Phi\in\mathrm{ENT}\,,
\label{eq:Lp_ent}
\end{align}
for some $1<p\leq 2$. 
Then $m\in B^{1/3}_{3p,\infty,\loc}(\Omega)$, that is,
\begin{align*}
\sup_{|h|>0}\frac{1}{|h|^{1/3}}\| m(\cdot +h)-m \|_{L^{3p}(U\cap(U -h))} <\infty\,,
\end{align*}
for all $U\subset\subset\Omega$.
\end{theorem}

The case $1<p\leq 4/3$ of Theorem~\ref{t:besov_reg} actually does not require the  boundary layer assumption \eqref{eq:layer}
and was proved in \cite{LP23},
as well as the
 converse implication, that $\dv\Phi(m)\in L^p_{\loc}$ 
if $m$ is a weak solution of \eqref{eq:eikonal}
 with $B^{1/3}_{3p,\infty,\loc}$ regularity.
The novelty with respect to \cite{LP23} resides in the way we treat boundary terms in the regularity estimate. 

\begin{proof}[Proof of Theorem~\ref{t:Lpent_tangent} from Theorem~\ref{t:besov_reg}]
Under the assumptions of Theorem~\ref{t:Lpent_tangent} 
\color{black}
we can apply \cite[Theorem~1.8]{LP23}
to deduce that 
 $\dv\Phi(m)\in L^p(\Omega)$ for all odd  $\Phi\in\mathrm{ENT}$
 (see Remark~\ref{r:oddENT} for the representation of odd entropies corresponding to the statement in \cite{LP23}).
\color{black}
 Thus we may apply Theorem~\ref{t:besov_reg} and (noting also that $m$ is $C^1$ in the boundary layer) obtain that 
$m\in B^{1/3}_{3p,\infty}(\Omega_\delta)$.
If $p=2$, then \cite[Theorem~1.1]{LLP25} shows that $\dv\Phi(m)=0$ for all $\Phi\in\mathrm{ENT}$, 
and \cite[Theorem~1.2]{JOP02} implies that $\Omega$ is a disk and $m$ a vortex.
If $p=1.95> (47+\sqrt{533})/36$ and $\Omega$ is a disk, then \cite[Theorem~1.5]{LLP25} shows that $m$ is a vortex.
\end{proof}

\begin{remark} The methods of \cite{LLP25} 
can actually be improved to show a stronger version of Theorem~\ref{t:Lpent_tangent}, namely, if \textcolor{black}{$\dv\Sigma_k(m)\in L^p(\Omega_\delta)$ for $k=1, 2$} and $p=1.95$
then $\Omega$ must be a disk, and $m$ a vortex.
The necessary modifications would be rather lengthy, that is why we only provide the short but weaker version stated here.
\end{remark}


\paragraph{Acknowledgments.}
 XL received support from ANR project ANR-22-CE40-0006. 
 AL  was supported in part by NSF grant DMS-2406283. 
 GP was supported in part by NSF grant DMS-2206291.

\section{Proof of Theorem~\ref{t:besov_reg}}\label{s:proof_besov}

The proof of Theorem~\ref{t:besov_reg} 
relies on a kinetic formulation of the finite-entropy condition \eqref{eq:finite_entropy} and
on compensation identities inspired from \cite{GP13} in the context of scalar conservation laws and their adaptation to the eikonal equation in \cite{GL20}.

\subsection{A kinetic formulation}

If $m$ is a finite-entropy solution of \eqref{eq:eikonal}-\eqref{eq:finite_entropy},
 the indicator function
\begin{align}\label{eq:chi}
\chi(s,x)=\mathbf 1_{m(x)\cdot e^{is}>0}\,,\quad s\in\T=\R/2\pi\Z,\; x\in\Omega\,,
\end{align} 
satisfies the kinetic formulation
\begin{align}\label{eq:kin}
e^{is}\cdot\nabla_x\chi 
=\Theta \quad\text{in }\mathcal D'(\T\times\Omega)\,,
\end{align}
for some $\Theta = \partial_s\sigma(s,x)$ with $\sigma\in\mathcal{M}(\T\times\Omega)$ \cite{GL20}, see also \cite{JP01}.

\color{black}
\begin{remark}\label{r:oddENT}
This actually requires \eqref{eq:finite_entropy} only 
for entropies of the form
\begin{align*}
\Phi^\psi(z)
&
=\int_{\T}\mathbf 1_{z\cdot e^{is}>0}\psi(s)e^{is}\, ds
=\frac 12 \int_{\T}\mathbf 1_{z\cdot e^{is}>0}(\psi(s)+\psi(\pi+s))e^{is}\, ds\,,
\end{align*}
for $\psi\in C^1(\T)$ with $\int_\T \psi(s)e^{is}\, ds =0$, 
see \cite[\S~3.1]{GL20}.
Any such $\Phi^\psi$ is odd, and reciprocally, any odd $\Phi\in\mathrm{ENT}$
 can be represented as $\Phi=\Phi^\psi$.
To see this, define
 $\psi(s)=\textcolor{black}{-}\frac 12 e^{is}\cdot\frac{d}{ds}\Phi(ie^{is})$.
The oddness of $\Phi$ implies that $\psi$ is $\pi$-periodic,
and one then checks that $\Phi-\Phi^\psi$ has zero derivative along $\mathbb S^1$, hence is zero, being constant and odd.
\end{remark}
\color{black}

If $m$ further satisfies the $L^p$ entropy production property \eqref{eq:Lp_ent}, 
then the measure $\sigma$ has an explicit form (see \cite{LP23,marconi21})
in terms of
$\theta\colon\Omega\to\T$  such that $m=e^{i\theta}$,
given by
\begin{align}\label{eq:sigma}
\sigma(s,x)
=
\left(\delta_{\theta(x)+\frac\pi 2}(s) + \delta_{\theta(x)-\frac\pi 2}(s)\right) F(x)
\quad\text{for some }
F\in L^p(\Omega)\,.
\end{align}

\subsection{A coercive quantity}

We can use the indicator function \eqref{eq:chi} to build a quantity that controls finite differences of $m$.
We denote by $T^h$ and  $D^h$ the translation and finite difference operators in the $x$ variable:
\begin{align*}
	T^h\! f(x)=f(x+h),\quad D^h\! f =T^h\! f-f\,,
\end{align*}
for any function $f(x)$,
and we will frequently use the identity
\begin{align}\label{eq:Dh_prod}
	D^h(fg)=f D^h g +T^h g D^h f\,.
\end{align}
For any
$\varphi\in L^1(\T)$ odd and $\pi$-periodic,
and for the indicator function \eqref{eq:chi},  we define the function
\begin{align}\label{eq:Delta}
	\Delta^{\varphi,\chi}(x,h)
	=\frac 12 \int_{\T^2} \varphi(t-s)D^h\chi(t,x) D^h\chi(s,x)\sin(t-s)\, dtds,
\end{align}
on $\Omega^*=\lbrace (x,h)\in \Omega\times\R^2\colon x+h\in\Omega\rbrace$.
This can be rewritten as
\begin{align*}
	\Delta^{\varphi,\chi}(x,h)=\frac 12 \Xi^\varphi(m(x),T^{h}m(x))\,,
\end{align*}
where $\Xi^\varphi\colon \mathbb S^1\times\mathbb S^1\to\R$ is given by
\begin{align*}
	\Xi^\varphi(z_1,z_2)
	&
	=\int_{\mathbb S^1\times\mathbb S^1}\varphi(t-s)\sin(t-s)
	\Big(
	\mathbf 1_{e^{it}\cdot z_2 >0}-\mathbf 1_{e^{it}\cdot z_1>0}\Big)
	\nonumber
	\\
	&
	\hspace{9em}
	\cdot
	\Big(
	\mathbf 1_{e^{is}\cdot z_2 >0}-\mathbf 1_{e^{is}\cdot z_1>0}\Big)
	\, dt ds\,.
	\label{eq:Xivarphi}
\end{align*}
For a large class of odd $\pi$-periodic functions $\varphi$,
this quantity $\Xi^\varphi(z_1,z_2)$  
provides control on the distance $|z_1-z_2|$.

\begin{lemma}\label{l:coercXivarphi}
	If $\varphi\colon\R\to\R$ is odd, $\pi$-periodic 
	and $\varphi \geq 0$ on $(0,\pi/2)$, then we have
	\begin{align*}
		&
		\Xi^\varphi(z_1,z_2) \gtrsim 
		\omega_\varphi(|z_1-z_2|)\,,
		\quad
		\text{where }
		\omega_\varphi(t)=t \int_0^{\frac t 4} s\varphi(s)\, ds\,,
		\quad\forall t\in [0,2]\,.
	\end{align*}
\end{lemma}
\begin{proof}[Proof of Lemma~\ref{l:coercXivarphi}]
	This follows essentially from the calculations in \cite[Lemma~3.8]{GL20}, which we briefly recall here.
	Thanks to the invariance $\Xi^\varphi(e^{i\alpha}z_1,e^{i\alpha}z_2)=\Xi^\varphi(z_1,z_2)=\Xi^\varphi(z_2,z_1)$ for all $\alpha\in\R$ and $z_1,z_2\in\mathbb S^1$,
	it suffices to consider the case $z_1=e^{-i\beta}$, $z_2=e^{i\beta}$ for some $\beta\in [0,\pi/2]$.
	Then the computations in \cite[Lemma~3.8]{GL20} give
	\begin{align*}
		\frac 18 \Xi^{\varphi}(e^{-i\beta},e^{i\beta})
		=
		\left\lbrace
		\begin{aligned}
			&
			\int_0^{2\beta}\varphi(t)(2\beta - t)\sin(t)\, dt
			&
			\quad\text{if }\beta\in [0,\pi/4]\,,
			\\
			&
			\int_0^{\frac\pi 2}\varphi(t)\min(2\beta - t,\pi - 2t )\sin(t)\, dt
			&
			\quad\text{if }\beta\in [\pi/4,\pi/2]\,.
		\end{aligned}
		\right.
	\end{align*}
	Since all factors in the integrands are nonnegative,
	using that 
	\begin{align*}
		&
		2\beta-t
		\geq \beta
		\quad
		\text{ for }
		0\leq t\leq \beta\,,
		\\
		&
		\sin(t)
		\geq 2t/\pi
		\quad
		\text{ for }0\leq t \leq \pi/2,
		\\
		\text{and }
		\quad
		&
		\min(2\beta - t,\pi - 2t )
		\geq \pi/4
		\quad
		\text{ for }
		0\leq t\leq\pi/4\leq\beta
		\,,
	\end{align*}
	we obtain
	\begin{align*}
		\Xi^{\varphi}(e^{-i\beta},e^{i\beta})
		\gtrsim \beta \int_0^{\frac\beta 2}\varphi(t) t\, dt\,.
	\end{align*}
	Noting that
	$|e^{-i\beta}-e^{i\beta}|=2\sin\beta \leq 2\beta$,
	we deduce the claimed estimate.
\end{proof}

\subsection{A compensation identity}

The quantity $\Delta^{\varphi,\chi}$ defined in \eqref{eq:Delta} satisfies a compensation identity, essentially established in \cite[Lemma~3.9]{GL20}.

\begin{lemma}\label{l:comp_id}
Let $\Omega\subset\R^2$, $\T=\R/2\pi\Z$,
$\chi\in L^\infty(\T\times \Omega)$
such that the distribution $\Theta\in \mathcal D'(\T\times\Omega)$ defined by
\begin{align*}
e^{is}\cdot\nabla_x \chi =\Theta
\quad\text{ in }\mathcal D'(\T\times\Omega)\,,
\end{align*}
is of the form $\Theta=\partial_s\sigma$ for some $\sigma\in L^1(\Omega;\mathcal M(\T))$.
Then, for any
 $\varphi\in C^1(\T)$ odd and $\pi$-periodic, the function
 $\Delta^{\varphi,\chi}$ defined in \eqref{eq:Delta}
satisfies
\begin{equation}\label{eq:comp_id}
\frac{d}{d\tau}\Delta^{\varphi,\chi}(x,\tau\mathbf e_1)
=I^\tau(x) +\dv_x A^\tau (x)\,,
\end{equation}
in $\mathcal D'(\lbrace (x,\tau)\in\Omega\times\R \colon x+\tau \mathbf e_1\in\Omega\rbrace )$,
where $I^\tau$ and  $A^\tau=(A_1^\tau,A_2^\tau)$ are given by
\begin{align*}
I^\tau(x)
&
=
\int_{\T^2} (\Theta +  T^{\tau\mathbf e_1}\Theta)(s,x) \varphi(t-s)D^{\tau\mathbf e_1}\chi(t,x)\sin(t) \, ds dt\,
\\
&
\quad
-D^{\tau \mathbf e_1}\bigg[
 \int_{\T^2} \Theta(s,x) \varphi(t-s)\chi(t,x)\sin (t) \, ds dt\bigg]\,,
\\
A_1^\tau(x)
&=
\int_{\T^2} \varphi(t-s)\sin(t)  \cos(s) 
T^{\tau\mathbf e_1}\chi(t,x)D^{\tau\mathbf e_1}\chi(s,x) \, ds dt\, ,
\nonumber
\\
A_2^\tau(x)
&
=\int_{\T^2} \varphi(t-s)\sin(t)\, \sin(s)\,
 \chi(t,x)D^{\tau\mathbf e_1}\chi(s,x)\, dsdt\,.
\end{align*}
In the expression of $I^\tau$, 
integrals $\int\Theta(s,\cdot )\varphi(t-s)\,ds$  should be understood in the sense of distributions, that is,
$\int\Theta(s,\cdot)\textcolor{black}{\varphi}(t-s)\,ds=\int\varphi'(t-s)\sigma(ds,\cdot)\in L^1(\Omega)$.
\end{lemma}

\begin{proof}[Proof of Lemma~\ref{l:comp_id}]
Assume first that, in addition to the above, $\chi$ is $C^1$ with respect to $x$ and $\nabla_x\chi\in L^\infty(\T\times\Omega)$.
Note that this implies that $\Theta$ is $C^0$ with respect to $x$, 
in the sense that $\int\Theta(s,\cdot)\psi(s)\,ds =\int \psi(s)e^{is}\cdot\nabla_x\chi(s,\cdot)\, ds$ is continuous for any $\psi\in C^1(\T)$.
Then the calculations of \cite[Lemma~3.9]{GL20} give the validity of \eqref{eq:comp_id} pointwise, see also \cite[Lemma~2.6]{BGLS25} which uses the same notation as here.

For a general  $\chi\in L^{\infty}(\T\times\Omega)$, one may therefore apply \eqref{eq:comp_id} to
a regularization $\chi_\e=\chi *_x \rho_\e$ by convolution with respect to the $x$ variable,
and the corresponding $\Theta_\e=\partial_s\sigma_\e$.
Then, for a fixed $\Omega'\subset\subset\Omega$, 
we have that $\chi_\e$ is uniformly bounded, 
  $\chi_\e\to\chi$ a.e., 
and for any test function $\psi\in C^1(\T)$, 
the functions $g_\e =\int_\T \psi(s)\sigma_\e(ds,\cdot)$ converge to 
$g=\int_\T \psi(s)\textcolor{black}{\sigma}(ds,\cdot)$ strongly in $L^1(\Omega')$.
One can therefore, by dominated convergence, let $\e\to 0$  in the distributional formulation of \eqref{eq:comp_id} for $\chi_\e$, and conclude that \eqref{eq:comp_id} holds for
$\chi$ in the sense of distributions.
\end{proof}

\subsection{Estimates for the right-hand side of the compensation identity}

We give estimates on $A^\tau$ and $I^\tau$ appearing in the right-hand side of \eqref{eq:comp_id}. For $A^\tau$ we simply invoke \cite[Lemma~2.7]{BGLS25}, which we restate here.

\begin{lemma}[{\cite{BGLS25}}]\label{l:estim_A}
Let $m\colon\Omega\to\mathbb S^1$ measurable and $\chi(s,x)=\mathbf 1_{e^{is}\cdot m(x) >0}$. 
Then the vector field $A^\tau$ in \eqref{eq:comp_id} satisfies
\begin{align*}
|A^\tau|\lesssim \|\varphi\|_{L^1(\T)} \, |D^{\tau \mathbf e_1} m|
\qquad\text{a.e. in }\Omega\cap(\Omega-\tau\mathbf e_1)\,.
\end{align*}
\end{lemma}


For $I^\tau$, we establish the following estimate.

\begin{lemma}\label{l:estim_I_C2}
	Let $m\colon\Omega\to\mathbb S^1$ such that $\chi(s,x)=\mathbf 1_{e^{is}\cdot m(x) >0}$
	satisfies \eqref{eq:kin}
	with $\Theta=\partial_s\sigma$ and $\sigma$ given by \eqref{eq:sigma}.
	Assume that the odd $\pi$-periodic function $\varphi$ in \eqref{eq:Delta} is $C^1$ and satisfies
	$|\varphi(t)|\lesssim |t|^\gamma$ for some $\gamma > 0$ and all $t\in (0,\pi/2]$.
	Then, for any $\eta\in C_c^1(\Omega;[0,1])$, 
	the function $I^\tau$ in \eqref{eq:comp_id} satisfies
	\begin{align*}
		\int_{\Omega} I^\tau\,\eta^2\, dx
		&
		\lesssim  
		\|\eta^2
		|D^{\tau\mathbf e_1}m|^{\gamma}
		\big\|_{L^{p'}(\Omega)} 
		\|F\|_{L^p(\Omega)}
		\\
		&
		\quad
		+ 
		|\tau | 
		\|\varphi'\|_{L^1(\T)}
		\|\nabla \eta\|_\infty 
		\|F\|_{L^1(\Omega)}
		\,,
	\end{align*}
	for $|\tau|\leq \dist(\supp(\eta),\partial\Omega)$.
\end{lemma}

\begin{proof}[Proof of Lemma~\ref{l:estim_I_C2}]
Recalling the expression of $I^\tau$ in \eqref{eq:comp_id},
plugging in $\Theta=\partial_s\sigma$ as in \eqref{eq:sigma}
and using that $\varphi'$ is $\pi$-periodic, 
we have
\begin{align*}
I^\tau
&
=I^\tau_1 + I^\tau_2
-D^{\tau \mathbf e_1}I_3\,,
\\
I^\tau_1(x)
&=
2F(x)\int_{\T}  \varphi'(t-\theta(x)-\frac\pi 2)D^{\tau\mathbf e_1}\chi(t,x)\sin(t) \, dt\,,
\\
I^\tau_2(x)
&=
2F(x+\tau\mathbf e_1)\int_{\T}  
\varphi'(t-\theta(x+\tau\mathbf e_1)-\frac\pi 2)
D^{\tau\mathbf e_1}\chi(t,x)\sin(t) \, dt\,,
\\
I_3(x)
&=
2F(x)\int_{\T}  \varphi'(t-\theta(x)\textcolor{black}{-\frac \pi 2})\chi(t,x)\sin (t) \, ds dt
\,.
\end{align*}
Recalling that  $\chi(t,x)=\mathbf 1_{e^{it}\cdot m(x)>0}$,
we rewrite these as
\begin{align*}
 I^\tau_1(x)
&
=
2F(x) \Big(G(im(x),m(x+\tau \mathbf e_1))-G(im(x),m(x)) \Big) 
 \,,
\\
 I^\tau_2(x)
&
=
2F(x+\tau \mathbf e_1) \Big(G(im(x+\tau \mathbf e_1),m(x+\tau \mathbf e_1))-G(im(x+\tau \mathbf e_1),m(x)) \Big) \,,
\\
I_3(x)
&=
2F(x) G(im(x),m(x))
\,,
\end{align*}
where
\begin{align}\label{eq:G}
G(e^{i\alpha},e^{i\theta})=\int_{\theta-\frac\pi 2}^{\theta +\frac\pi 2}\varphi'(t-\alpha)\sin(t)\, dt\,.
\end{align}
Note that $G\colon\mathbb S^1\times\mathbb S^1\to\R$ is well-defined: its expression does not depend on the choice of $\textcolor{black}{\alpha}$ modulo $2\pi$ since $\varphi$ is $2\pi$-periodic, nor on the choice of $\theta$ modulo $2\pi$ since 
$\int_0^{2\pi}\varphi'(t-s)\sin(t)\,dt =0$  as a consequence of the $\pi$-periodicity of $\varphi$.
Using that 
\begin{align*}
\frac{d}{dt}\Big[G(e^{i\alpha},e^{it})\Big]&
=2\varphi'(t+\frac\pi 2 - \alpha)\sin(t +\frac\pi 2)\,,
\end{align*} 
we find,
for any $z=e^{i\theta}\in\mathbb S^1$,
\begin{align*}
\frac{G(iz,e^{i\beta}z)-G(iz,z)}{2}
&
=\frac 12 \int_{0}^\beta \frac{d}{dt}\Big[ G(e^{i(\theta+\frac\pi 2)},e^{i(\theta+t)})\Big]
\, dt
\\
&
=\int_0^\beta
\varphi'(t)\sin(\theta+t+\frac\pi 2)\, dt
\\
&
=
\varphi(\beta)\cos(\theta+\beta)+\int_0^\beta
\varphi(t)\sin(\theta+t)\, dt\, ,
\end{align*}
and can therefore further rewrite $I_1^\tau$ and $I_2^\tau$ as
\begin{align}
 I^\tau_1(x)
&
=
4F(x) 
\,
H(m(x), m(x+\tau\mathbf e_1)/m(x))\,,\nn
\\
 I^\tau_2(x)
&
=
4F(x+\tau \mathbf e_1) \,
H(m(\textcolor{black}{x+\tau\mathbf e_1}), m(x+\tau\mathbf e_1)/m(x))\,,
\nonumber
\\
\text{where }
H(e^{i\theta},e^{i\alpha})
&
=\varphi(\alpha)\cos(\theta+\alpha)+\int_0^{\alpha}
\varphi(t)\sin(\theta+t)\, dt \,.
\label{eq:H}
\end{align}
Here $z_1/z_2=z_1\bar z_2$ for $z_1,z_2\in\mathbb S^1$ is the division of complex numbers.
Note that $H$ is well-defined on $\mathbb S^1\times\mathbb S^1$ since its expression is $2\pi$-periodic with respect to both $\theta$ and $\alpha$, 
thanks to the $\pi$-periodicity of $\varphi$.
Summarizing, we have
\begin{align*}
 I^\tau_1(x)
&
=
4F(x) 
\,
H(m(x), m(x+\tau\mathbf e_1)/m(x))\,,
\\
 I^\tau_2(x)
&
=
4F(x+\tau \mathbf e_1) \,
H(m(\textcolor{black}{x+\tau\mathbf e_1}), m(x+\tau\mathbf e_1)/m(x))\,,
\\
I_3(x)
&=
2F(x)\, G(im(x),m(x))
\,,
\end{align*}
with $G$ as in \eqref{eq:G} and $H$ as in \eqref{eq:H}.
The function $G$ satisfies
\begin{align*}
|G|\lesssim \|\varphi'\|_{L^1}\,.
\end{align*}
And, recalling  $|\varphi(t)|\lesssim |t|^\gamma$, the function $H$ satisfies
\begin{align*}
|H({\color{black}w},e^{i\alpha})|\lesssim |\alpha|^{\gamma}\,,
\quad
\text{
hence }
|H(\textcolor{black}{w},z'/z)|\lesssim |z'/z-1|^{\gamma}=|z'-z|^{\gamma}\,,
\end{align*}
for all ${\color{black}w},z,z'\in\mathbb S^1$.
Thus we have
\begin{align*}
|I_1^\tau|
&
\lesssim |F| \, |D^{\tau\mathbf e_1}m|^{\gamma}\,,
\\
|I_2^\tau|
&
\lesssim T^{\tau\mathbf e_1}|F|\,|D^{\tau\mathbf e_1}m|^{\gamma}\,,
\\
|I_3|
&
\lesssim \|\varphi'\|_{L^1(\T)} |F|\,.
\end{align*}
Integrating  $I^\tau$ against $\eta^2\, dx$,
performing a discrete integration by parts using \eqref{eq:Dh_prod} in the term involving $D^{\tau \mathbf e_1} \textcolor{black}{I_3}$,
and using Hölder's inequality, we obtain the claimed estimate.
\end{proof}


\subsection{Regularity estimate}

\begin{proof}[Proof of Theorem~\ref{t:besov_reg}]
 Fix some $\Omega'\subset\subset U \subset\subset \Omega$, and a smooth cut-off function $\eta\in C_c^{\infty}(\Omega)$ such that $\mathbf 1_{U}\leq \eta\leq \mathbf 1_\Omega$.

%
%
Thanks to the assumption \eqref{eq:Lp_ent}, the indicator function $\chi$ satisfies the kinetic formulation \eqref{eq:kin} with $\Theta=\partial_s\sigma$ as in \eqref{eq:sigma}.
For any odd $\pi$-periodic $\varphi\in C^1(\R)$ we may therefore apply the compensation identity 
\eqref{eq:comp_id}.
Let $h=\tau\mathbf e_1$ with 
$0<\tau<r_0 = \min\{\dist(\supp (\eta),\partial\Omega), \dist(\Omega', \partial U)\}$, then we have
\begin{align}\label{eq:estim_Delta}
		\int_{\Omega} \Delta^{\varphi,\chi}(x,h)\,\eta^2\,dx = \int_0^\tau\bigg(\int_{\Omega}I^{\tau'}\eta^2\,dx - 2\int_{\Omega} (A^{\tau'}\cdot\nabla\eta)\,\eta\,dx
\bigg)\,d\tau'.
	\end{align}
We assume also that
 $|\varphi(t)|\lesssim |t|^\gamma$ for all $t\in (0,\pi/2]$.
 Then, from Lemma~\ref{l:estim_I_C2} and Lemma~\ref{l:estim_A}, we have
\begin{align*}
		\int_0^\tau\int_{\Omega}I^{\tau'}\eta^2\,dx\,d\tau'
		&
		\lesssim  
		\tau \sup_{\tau'\leq\tau}\|\eta^2
		|D^{\tau'\mathbf e_1} m|^{\gamma}
		\big\|_{L^{p'}(\Omega)} 
		\|F\|_{L^p(\Omega)}\nonumber
		\\
		&
		\quad
		+ 
		\tau^2 
		\|\varphi'\|_{L^1(\T)}
		\|\nabla \eta\|_\infty 
		\|F\|_{L^1(\Omega)}
		\,,
		\\
		 \int_0^\tau\int_{\Omega} | A^{\tau'}\cdot\nabla\eta |\,\eta\,dx\,d\tau' 
		 &
		 \lesssim \tau\|\varphi\|_{L^1(\T)} \, \sup_{\tau'\leq\tau}\int_{\Omega}|D^{\tau' \mathbf e_1} m|\, |\nabla\eta|\,\eta\,dx\,. 
\\
&
\lesssim \tau^2 \|\varphi\|_{L^1(\T)}\|\nabla\eta\|_\infty\|\nabla m\|_{L^1(\Omega\setminus\overline{\Omega'})}\,,
	\end{align*}
where the last inequality follows from the facts that 
$\supp(\eta |\nabla\eta|)+B_{r_0}\subset\Omega\setminus\overline{\Omega'}$ and $m\in W^{1,1}(\Omega\setminus\overline{\Omega'})$.
Plugging these inequalities into \eqref{eq:estim_Delta}
gives
\begin{align*}
\int_{\Omega} \Delta^{\varphi,\chi}(x,h)\,\eta^2\,dx
&
\lesssim
		\tau 
		\sup_{\tau'\leq\tau}
		\big\|\eta^2
		|D^{\tau'\mathbf e_1} m|^{\gamma}
		\big\|_{L^{p'}(\Omega)} 
		\|F\|_{L^p(\Omega)}\nonumber
		\\
		&
		\quad
		+ 
		\tau^2 
		\|\varphi'\|_{L^1(\T)}
		\|\nabla \eta\|_\infty 
		\|F\|_{L^1(\Omega)}\\
		&\quad+\tau^2\|\varphi\|_{L^1(\T)} \, \|\nabla\eta\|_{\infty}\,\|\nabla m\|_{L^1(\Omega\setminus\overline{\Omega'})}\,. 
	\end{align*}
This estimate is valid for any $\varphi$ odd, $\pi$-periodic and $C^1$ such that $|\varphi(t)|\lesssim |t|^\gamma$ for $t\in (0,\pi/2]$.
By approximation, this is also true for 
 $\varphi\colon \R\to\R$  odd, $\pi$-periodic, and $C^1$ on $(0,\pi)$ such that $\varphi\geq 0$ on $(0,\pi/2)$ and $\varphi(t) = t^\gamma$ for $t\in(0,\pi/4)$, with $\gamma = 3p-3$.
(If $\gamma>1$ this $\varphi$ is $C^1$ on $\R$ and no approximation is needed.) 
 For such a choice, Lemma~\ref{l:coercXivarphi} ensures that
\begin{equation*}
		\Delta^{\varphi,\chi}(x,h)\geq c|D^{h} m(x)|^{3+\gamma} = c|D^{h} m(x)|^{3p}\,,
	\end{equation*}
so we deduce
\begin{align*}
\int_{\Omega} |D^{\tau\mathbf e_1} m|^{3p}\,\eta^2\,dx
&
\lesssim
		\tau 
		\sup_{\tau'\leq\tau}
		\big\|\eta^2
		|D^{\tau'\mathbf e_1} m|^{\gamma}
		\big\|_{L^{p'}(\Omega)} 
		\|F\|_{L^p(\Omega)}\nonumber
		\\
		&
		\quad
		+ 
		\tau^2 
		\|\varphi'\|_{L^1(\T)}
		\|\nabla \eta\|_\infty 
		\|F\|_{L^1(\Omega)}\\
		&\quad+\tau^2\|\varphi\|_{L^1(\T)} \, \|\nabla\eta\|_{\infty}\,\|\nabla m\|_{L^1(\Omega\setminus\overline{\Omega'})}\,. 
	\end{align*}
	Note that $\gamma\, p'= 3p$. 
	Taking a supremum over $\tau$ in the left hand side, and using Young's inequality and the fact that $\eta\leq 1$ in the first term on the right hand side, we obtain
	\begin{align*}
		\sup_{\tau'\leq\tau}\int_{\Omega}|D^{\tau'\mathbf e_1} m|^{3p}\,\eta^2\,dx &
		\leq \frac{1}{2}\sup_{\tau'\leq\tau}\int_{\Omega}
		|D^{\textcolor{black}{\tau'\mathbf e_1}} m|^{3p}\,\eta^2\,dx+
		C\,\tau^p 
		\|F\|_{L^p(\Omega)}^p\nonumber
		\\
		&
		\quad
		+ 
		C\tau^2 
		\|\varphi'\|_{L^1(\T)}
		\|\nabla \eta\|_\infty 
		\|F\|_{L^1(\Omega)}\\
		&\quad+C\tau^2\|\varphi\|_{L^1(\T)} \, \|\nabla\eta\|_{\infty}\,\|\nabla  m\|_{L^1(\textcolor{black}{\Omega\setminus\overline{\Omega'}})}\,, 
	\end{align*}
hence, recalling $1<p\leq 2$,
	\begin{align*}
\int_{\Omega}|D^{h} m|^{3p}\,\eta^2\,dx \leq C |h|^p\quad\text{for }h=\tau\mathbf e_1,\;0<\tau<r_0\,,
	\end{align*}
where the constant $C>0$ depends on $p$, $\|F\|_{L^p(\Omega)}$,  $\Omega$, $\Omega'$, $U$ and $m_{\lfloor \Omega\setminus\overline{\Omega'}}$.
Rotating the variable, this holds for every $0<|h|<r_0$,
and since $|D^h m|\leq 2$, this holds for every $|h|>0$, 
with a possibly larger constant $C>0$.
Recalling that $\eta\equiv 1$ on $U$, we deduce that $m\in B^{1/3}_{3p,\infty}(U)$.
\end{proof}

\bibliographystyle{acm}
\bibliography{AG}

\end{document}